\documentclass[a4paper,11pt]{article}

\usepackage[T1]{fontenc}

\usepackage{bbm}
\usepackage{amsmath, wrapfig}
\usepackage{dsfont, a4wide, amsthm, amssymb, amsfonts, graphicx}
\usepackage{fancyhdr, xspace, psfrag, setspace, supertabular, color, enumerate}
\usepackage{hyperref}

\newtheorem{theorem}{Theorem}[section]
\newtheorem{lemma}[theorem]{Lemma}
\newtheorem{proposition}[theorem]{Proposition}

\newtheorem*{conjecture*}{Conjecture}

\newtheorem*{theorem*}{Theorem}
\newtheorem{definition}[theorem]{Definition}
\newtheorem{question}[theorem]{Question}
\newtheorem*{question*}{Question}

\theoremstyle{definition}
\newtheorem{example}[theorem]{Example}

\def\e{\varepsilon}
\def\a{\alpha}
\def\b{\beta}

\def\ll{\lambda}
\def\s{\sigma}
\def\kk{\kappa}
\def\ce{\mathcal{E}}

\def\cp{\mathcal{P}}
\def\R{\mathbb{R}}

\def\Z{\mathbb{Z}}

\def\Z{\mathbb{Z}}
\def\E{\mathop{\mathbb{E}}}
\def\P{\mathbb{P}}

\def\so3{\mathrm{SO}(3)}
\def\fso3{\mathrm{FSO}(3)}

\let\eps\varepsilon

\DeclareMathOperator{\Var}{Var}
\DeclareMathOperator{\Cov}{Cov}

\title{The structure of sets with cube-avoiding sumsets}
\author{Thomas Karam\footnote{Mathematical Institute, University of Oxford. Email: \texttt{thomas.karam@maths.ox.ac.uk}. Supported by ERC Advanced Grant 883810.} \and Peter Keevash\footnote{Mathematical Institute, University of Oxford. Email: \texttt{keevash@maths.ox.ac.uk}. Supported by ERC Advanced Grant 883810.}
}

\begin{document}
\maketitle

\begin{abstract}
We prove that if $d \ge 2$ is an integer, $G$ is a finite abelian group, $Z_0$ is a subset of $G$ not contained in any strict coset in $G$, and $E_1,\dots,E_d$ are dense subsets of $G^n$ such that the sumset $E_1+\dots+E_d$ avoids $Z_0^n$ then $E_1, \dots, E_d$ essentially have bounded dimension. More precisely, they are almost entirely contained in sets $E_1' \times G^{I^c}, \dots, E_d' \times G^{I^c}$, where the size of $I \subset [n]$ is non-zero and independent of $n$, and $E_1',\dots,E_d'$ are subsets of $G^{I}$ such that the sumset $E_1'+\dots+E_d'$ avoids $Z_0^I$.
\end{abstract}

\section{Introduction}

An important direction in combinatorial number theory and geometry considers questions that are broadly of the following kind:~given two subsets $A$,$B$ of some ambient (abelian) group, what may be deduced about the structure of $A$,$B$ if the sumset $A+B$ is somehow constrained?

Among various constraints that may be considered for $A+B$, the most basic and most studied is bounding $|A+B|$ in terms of $|A|$ and $|B|$. For example, when $A=B$, there is a long line of research describing the structure of finite sets $A \subset \Z$ with $|A+A| \le K|A|$ for some fixed $K$. Here the landmark result is Freiman's theorem, which shows that such $A$ must be contained in some multidimensional arithmetic progression $P$ with $\dim(P)$ and $|P|/|A|$ both bounded by a constant depending only on $K$. Further milestones in this direction are an extension to general (not necessarily abelian) groups by Breuillard, Green and Tao \cite{Breuillard Green Tao} and a resolution of Marton's Conjecture (aka the Polynomial Freiman-Ruzsa Conjecture) on a polynomial quantitative improvement for  abelian groups with bounded torsion by Gowers, Green, Manners and Tao \cite{Gowers Green Manners Tao}.

It is also natural to constrain $A+B$ not by placing an upper bound on its size, but by requiring that it avoids a specific structured set. For instance, a theorem of S\'ark\"ozy \cite{Sarkozy} provides upper bounds on the size of a subset $A \subset [n]$ such that $A-A$ does not contain any prime integer, and Green \cite{Green} recently obtained a related result for shifted primes. Another well-studied example is that of subsets $A \subset [n]$ such that $A-A$ does not contain any perfect square, for which bounds were obtained by S\'ark\"ozy \cite{Sarkozy} and Furstenberg \cite{Furstenberg}, with a more recent improvement by Bloom and Maynard \cite{Bloom Maynard}. 
 
Our own focus in the present paper will be on high-dimensional subsets $E,F$ of a large power $G^n$ of some fixed finite abelian group $G$, 
satisfying the constraint that for some fixed $Z_0 \subset G$ the sumset $E+F$ avoids the power $Z_0^n$.
To illustrate this, we consider two subsets $A,B$ of $\Z^n$ with  constraints 
according to residues modulo some integer $N$. Here we give two examples, the first being trivial, 
and the second already capturing the main difficulties of the question.
Suppose first that $N=2$ and we require that there is no $(a,b) \in A \times B$ with $a_i+b_i$ even for every $i \in [n]$. 
Equivalently, projecting $A,B$ modulo $2$ to subsets $E,F$ of $\Z_2^n$,
we require that $E+F$ does not contain $\{0\}^n \in \Z_2^n$, or equivalently that $F$ is disjoint from $-E$.
In this example, we do not obtain any non-trivial structure for $E$ and $F$.

Now suppose that $N>2$ and $A+B$ avoids $(N\Z \cup (N\Z+1))^n$. 
Equivalently, projecting $A,B$ modulo $N$ to subsets $E,F$ of $\Z_N^n$,
we require $E+F$ to be disjoint from $\{0,1\}^n$. 
Here it is not clear what structure this imposes on $E$ and $F$. 
Furthermore, one may hope that an answer to this question may lead 
to progress on the longstanding Additive Basis Conjecture of Jaeger, Linial, Payan and Tarsi
(reported by Alon, Linial and Meshulam \cite{Alon Linial Meshulam}),
which is equivalent to the statement that for any prime $p$
there is a constant $C=C(p)$ such that for any
invertible linear maps $A_i \in GL(n,\Z_p)$ for $1 \le i \le C$
we have $A_1(\{0,1\}^n) + \cdots A_C(\{0,1\}^n) = \Z_p^n$.

In a companion paper \cite{Karam Keevash} we solve the corresponding continuous extremal problem
for $E$ and $F$ in the torus $(\R/\Z)^n$, which can be rephrased in terms of extremal expansions by cubes:
given $m, \ell \in (0,1)$ we describe the extremal examples (and show stability) for minimising
$\mu(E + [0,\ell]^n)$ among all $E \subset  (\R/\Z)^n$ with $\mu(E)=m$. One application of our extremal result
is another proof of the known bound $C < O(\log n)$ for the Additive Basis Conjecture.
The present paper concerns structural properties of $E,F$ with $(E+F) \cap \{0,1\}^n = \emptyset$
that have density bounded below but may be quite far from any extremal example.

\subsection{Results}
 
Next we make three basic observations
that motivate our structural result for subsets $E,F$ of $\Z_N^n$ 
such that $E+F$ avoids $\{0,1\}^n$ (or more generally $Z_0^n$ for some $Z_0 \subset \Z_N$).
\begin{enumerate}
\item Many natural examples are low-dimensional, e.g. $E=F= \{x \in \Z_3^n: x_1 = 1\}$.
\item Any example in some dimension can be extended to any higher dimension:
if $E',F'$ are subsets of $\Z_N^k$ such that $E'+F'$ avoids $\{0,1\}^k$ 
then we can extend $E',F'$ to subsets $E,F$ of $\Z_N^n$ for any $n>k$
simply by taking $E = E' \times \Z_N^{n-k}$ and $F = F' \times \Z_N^{n-k}$.
\item If $E+F$ avoids $\{0,1\}^n$ then so does $E'+F'$ for any $E' \subset E$ and $F' \subset F$.
\end{enumerate}
 
Conversely, our main result will state that any $E,F$ with $(E+F) \cap Z_0^n$ empty
can be approximated by sets obtained by these three moves
(starting with a low-dimensional set, extending to a higher dimension, taking arbitrary subsets).
We start with the formulation when $N$ is prime and $E+F$ avoids $Z_0^n$ 
for some $Z_0 \subset \Z_N$ with $|Z_0|>1$
(as illustrated above, if $|Z_0|=1$ then $E,F$ need not have any non-trivial structure).
Throughout, if $I$ is a subset of $[n]$ we write $I^c$ for the complement of $I$ in $[n]$. 

\begin{theorem} \label{Main result}
Let $p$ be a prime, let $\e > 0$, and let $Z_0 \subset \Z_p$ with $|Z_0|>1$. 
Then there exists $C=C(p,\e)$ such that for any subsets $E,F$ of $\Z_p^n$ with $(E+F) \cap Z_0^n = \emptyset$
there exist $I \subset [n]$ with $0<|I|<C$ and subsets $E',F'$ of $\Z_p^{I}$ satisfying 
\[ |E \setminus (E' \times \Z_p^{I^c})| \le \e |\Z_p^n|,
\quad |F \setminus (F' \times \Z_p^{I^c})| \le \e |\Z_p^n|,
\quad (E'+F') \cap Z_0^I = \emptyset. \]
\end{theorem}

We note that in Theorem \ref{Main result} we require the set $I$ of structured coordinates 
to have bounded size, be non-empty (for non-triviality) and be the same for both $E$ and $F$.
If instead we showed that $E$ and $F$ are structured with respect to some unrelated sets $I_E$ and $I_F$
then we could recover the formulation above by taking $I = I_E \cup I_F$, but in the converse direction 
we cannot allow unrelated sets (e.g.~if $I_E \cap I_F = \emptyset$ then we may have $E+F=\Z_p^n$).

Next we suppose that the modulus $N$ is not necessarily prime. 
Compared to the case of cyclic groups $\Z_N$ it will cost us little to tackle the case of finite abelian groups, 
so we shall present the corresponding generalisation of Theorem \ref{Main result} in the latter setting. 
Throughout, if $G$ is a finite abelian group, then we will say that a \emph{strict coset in $G$} 
is a set of the type $H + \{x\}$ where $H$ is a subgroup of $G$ with $H \ne G$, and $x$ is an element of $G$.

\begin{theorem} \label{Main result, groups case}

Let $G$ be a finite abelian group and let $\e > 0$. Then there exists $C=C(G,\e)$ such that 
if $Z_0 \subset G$ is not contained in any strict coset in $G$ 
and $E,F$ are subsets of $G^n$ with $(E+F) \cap Z_0^n = \emptyset$ 
then there exist $I \subset [n]$ with $0<|I|<C$ and subsets $E',F'$ of $G^{I}$ satisfying 
\[ |E \setminus (E' \times G^{I^c})| \le \e |G^n|,
\quad |F \setminus (F' \times G^{I^c})| \le \e |G^n|,
\quad (E'+F') \cap Z_0^I = \emptyset. \]
\end{theorem}

The assumption on $Z_0$ cannot be weakened,
as if $Z_0$ is contained in a coset of some strict subgroup $H$ of $G$
then under the projection $\pi:G \to G/H$ we have $|\pi(Z_0)| \le 1$,
which leads to a lack of structure similarly to the case $|Z_0|=1$
(see Example \ref{counterexample for two sets} for details).

Finally, we give a further extension of Theorem \ref{Main result, groups case} that considers several summands.

\begin{theorem} \label{Main result, groups case, several sets}

Let $d \ge 2$ be an integer, let $G$ be a finite abelian group, and let $\e > 0$.  
Then there exists $C=C(d,G,\e)$ such that 
if $Z_0 \subset G$ is not contained in any strict coset in $G$ 
and $E_1,\dots,E_d$ are subsets of $G^n$ with $(E_1+\dots+E_d) \cap Z_0^n = \emptyset$ 
then there exist $I \subset [n]$ with $0<|I|<C$ and subsets $E_1',\dots,E_d'$ of $G^{I}$ satisfying 
\[ |E_j \setminus (E_j' \times G^{I^c})| \le \e |G^n| \text{ for all } j \in [d],
\quad (E_1'+\dots+E_d') \cap Z_0^I = \emptyset. \]

\end{theorem}

In Section \ref{Section: The case of two subsets of F_p^n} we will prove Theorem \ref{Main result}. 
Section \ref{Section: Extending to several subsets of finite abelian groups} discusses 
the necessary modifications to prove its two successive generalisations, 
Theorem \ref{Main result, groups case} and Theorem \ref{Main result, groups case, several sets}. 
The final Section \ref{Section: Further discussion} contains some potential further generalisations and remaining open questions.

\subsection{Notation and conventions}

If $X$ is a finite set, $I$ is a subset of $[n]$, and $E$ is a subset of $X^I$, then we refer to the ratio $|E|/|X^I|$ as the \emph{density of $E$ inside $X^I$}. Usually there will be no ambiguity as to which set $X^I$ is being considered and we will denote this density by $d(E)$. If $y$ is an element of $X^I$ and $E$ is a subset of $X^n$, then we write $E_{I \to y}$ for the set of points $x \in E$ such that $x_i = y_i$ for every $i \in I$. Although the set $E_{I \to y}$ is a subset of $\Z_p^n$ rather than a subset of $\Z_p^{I^c}$, we can also view it as a subset of $\Z_p^{I^c}$, and the notation $d(E_{I \to y})$ will always refer to the density of $E_{I \to y}$ as a subset of $\Z_p^{I^c}$. If $I,J$ are disjoint subsets of $[n]$, $y$ and $z$ are elements of $X^I$ and $X^J$ respectively, and $E$ is a subset of $\Z_p^n$, then we write $E_{I \to y, J \to z}$ for the set $E_{I \cup J, w}$ where $w$ is the element of $X^{I \cup J}$ defined by $w_i = y_i$ for each $i \in I$ and $w_i = z_i$ for each $i \in J$.

\section{Two summands and prime modulus} \label{Section: The case of two subsets of F_p^n}

In the present section we prove Theorem \ref{Main result}. 
The two main technical ingredients of the proof are 
(i) a simultaneous regularity lemma for two sets, and
(ii) a result of Ha\k{z}\l{}a, Holenstein and Mossel on product space models of correlation.
We present the first ingredient in the first subsection, which is a minor modification of an existing result,
although we include a proof for the convenience of the reader.
In the second subsection we prove Theorem \ref{Main result}
assuming a lemma on summands in dense pseudorandom sets.
We then conclude the proof in the third subsection by using the second
ingredient mentioned above to prove this lemma.

\subsection{Pseudorandom sets and a simultaneous regularity lemma} 

In our context, a regularity lemma is a result on decomposing any subset of a product set $X^n$
into a bounded number of pieces, most of which are pseudorandom, in the following sense.

\begin{definition} 
Let $X$ be a finite set, let $r$ be a nonnegative integer and let $\b>0$. 
We say that a subset $E$ of $X^n$ is \emph{$(r,\b)$-pseudorandom} 
if for any subset $I$ of $[n]$ with size at most $r$ and every $y \in X^I$ we have $|d(E_{I \to y}) - d(E)| \le \b$. 
\end{definition}

The following formulation is similar to that in \cite[Lemma 3.2]{Keevash Lifshitz Long Minzer}, 
with the slight complication of simultaneously decomposing two sets;
the required modification of the proof is straightforward,
but we give the details for the convenience of the reader.

\begin{lemma} \label{Simultaneous regularity lemma} 
Let $X$ be a finite set, let $r$ be a nonnegative integer, and let $\b,\a>0$. 
Then there exists $C = \mathrm{Psr}_2(X,r,\b,\a)$ such that 
for any subsets $E,F$ of $X^n$ there exists $I \subset [n]$ with $0<|I|<C$
which simultaneously satisfies 
\begin{align} \P_{y \in X^I} (E_{I \to y} & \text{ not } (r,\b) \text{-pseudorandom}) \le \a, \label{first pseudorandomness inequality}\\
\P_{y \in X^I} (F_{I \to y} & \text{ not } (r,\b) \text{-pseudorandom}) \le \a. \label{second pseudorandomness inequality} \end{align} 
\end{lemma}

\begin{proof} 

We construct the set $I$ using an inductive process. 
Let us begin by ignoring the requirement that $I$ is non-empty
(which will be easy to also ensure).
For a positive integer $s$, at the $s$th iteration we proceed as follows, unless we have stopped before then.

Having obtained a set $I_s$, if the inequalities \eqref{first pseudorandomness inequality} and \eqref{second pseudorandomness inequality} hold for $I=I_s$ then we stop. 
Let us now assume that \eqref{first pseudorandomness inequality} fails 
(if instead \eqref{second pseudorandomness inequality} fails, then we proceed similarly with $E$ replaced by $F$). 
We then define the set \[E_s^{\mathrm{psr}}= \{y \in X^{I_s}: E_{I_s \to y} \text{ not } (r,\b) \text{-pseudorandom} \}.\] 
For each $y \in E_s^{\mathrm{psr}}$ we can find some $I_{s+1,y} \subset I_s^c$ 
with size at most $r$ satisfying \[|d(E_{I_s \to y, I_{s+1,y} \to z}) - d(E_{I_s \to y})| > \b\] for some $z \in I_{s+1,y}$.
To complete the inductive step we set \[I_{s+1} = I_s \cup \bigcup_{y \in E_s^{\mathrm{psr}}} I_{s+1,y}.\] 

We next show that the process terminates after a number of iterations that is bounded above depending on $p,r,\b,\a$ only. 
We will do so by an energy-increment argument. For each step $s$ of the induction we define the energies 
\[ \ce_s(E)  = \E_{y \in X^{I_s}} d(E_{I_s \to y})^2, \qquad
\ce_s(F)  = \E_{y \in X^{I_s}} d(F_{I_s \to y})^2. \]
The sets $I_s$ constitute an increasing sequence (with respect to inclusion), 
so the sequences $\ce_s(E)$ and $\ce_s(F)$ are nondecreasing by the Cauchy-Schwarz inequality. 
If some step $s$ of the induction involves $E$, then we may lower bound the difference \[\ce_{s+1}(E) - \ce_{s}(E) = \E_{y \in X^{I_s}} \left( \E_{z \in X^{I_{s+1} \setminus I_s}} d(E_{I_s \to y, I_{s+1} \setminus I_s \to z})^2 - d(F_{I_s \to y})^2 \right)\] by interpreting for every $y \in X^{I_s}$ the inner expectation as the variance of the variable $d(E_{I_s \to y, I_{s+1} \setminus I_s \to z})$ when $z$ is chosen uniformly at random in $X^{I_{s+1} \setminus I_s}$. For every $y \in E_s^{\mathrm{psr}}$, the variance of the variable $d(E_{I_s \to y, I_{s+1} \setminus I_s \to z})$ when $z$ is chosen uniformly at random in $X^{I_{s+1} \setminus I_s}$ is again (by the Cauchy-Schwarz inequality) at least the variance of the variable $d(E_{I_s \to y, I_{s+1,y} \to z})$ when $z$ is chosen uniformly at random in $X^{I_{s+1,y}}$, which is always at least $p^{-r} \b^2$ by definition of $E_s^{\mathrm{psr}}$.
Hence we obtain the lower bound \[\ce_{s+1}(E) - \ce_{s}(E) \ge p^{-r} \a \b^2,\] which is independent of $s$. 
Thus the sum of the energies $\ce_{s}(E) + \ce_{s}(F)$ increases by at least $p^{-r} \a \b^2$ at each step, and since it is always bounded above by $2$, the number of steps is at most $2 p^r \a^{-1} \b^{-2}$. At every step, the set $I_{s+1} \setminus I_s$ has size bounded above by $p^{|I_s|} r$.

To ensure that $I$ is non-empty, at the very first iteration we continue (rather than stop) regardless of whether \eqref{first pseudorandomness inequality} or \eqref{second pseudorandomness inequality} holds, and for that iteration the difference $(\ce_1(E) + \ce_1(F)) -  (\ce_0(E) + \ce_0(F))$ is lower bounded by $0$ rather than by $p^{-r} \a \b^2$. \end{proof}

\subsection{Reducing to pseudorandom summands}

Here we will reduce the structure theorem to the following proposition on pseudorandom dense summands $E$ and $F$,
showing that not only does $E+F$ meet $Z_0^n$ but $E \times F$ contains a positive proportion of the pairs with sum in $Z_0^n$. In what follows, 
we will continue to refer to subsets $E,F \subset \Z_p^n$ for simplicity of notation, but the upcoming considerations 
will in fact be applied to subsets $E_{I \to x'}$ and $F_{I \to y'}$ of $\Z_p^{I^c}$ 
obtained by applying Lemma \ref{Simultaneous regularity lemma}.

\begin{proposition} \label{Obtaining a non-empty intersection} 
Let $p$ be a prime, let $\e>0$ and let $Z_0$ be a subset of $\Z_p$ with $|Z_0|>1$.
Let $S = \{ (x,y) \in \Z_p^n \times \Z_p^n: x+y \in Z_0^n \}$.
Then there exist $\b>0$, $c>0$, and a positive integer $r$ such that 
whenever $E,F$ are $(r,\b)$-pseudorandom subsets of $\Z_p^n$ with density at least $\e$
we have $|(E \times F) \cap S| \ge c|S|$.
In particular, there exist $x \in E$, $y \in F$ satisfying $x+y \in Z_0^n$.
\end{proposition}

Let us now explain how Theorem \ref{Main result} follows from Lemma \ref{Simultaneous regularity lemma} and Proposition \ref{Obtaining a non-empty intersection}.

\begin{proof} [Proof of Theorem \ref{Main result}]
Let $\e>0$, and let $E,F$ be subsets of $\Z_p^n$ satisfying $(E+F) \cap Z_0^n = \emptyset$. 
First we consider the trivial case that one of $E$ or $F$ is sparse, say $d(E) \le \e$. 
Then taking $I=\{1\}$ (say), $E' = \emptyset$ and $F'= \Z_p$ satisfies all three required conclusions. 
Likewise if $d(F) \le \e$ by exchanging the roles of $E$ and $F$, so we may now assume both $d(E) \ge \e$ and $d(F) \ge \e$.

We then let $r$, $\b$ be the parameters given by Proposition \ref{Obtaining a non-empty intersection} for $p$,$Z_0$,$\e/2$. 
Applying Lemma \ref{Simultaneous regularity lemma} with $X=\Z_p$, these $r,\b$ and $\a = \e/2$, 
we obtain a non-empty subset $I$ of $[n]$ with size at most $\mathrm{Psr}_2(\Z_p,r, \b, \e/2)$ 
satisfying both \eqref{first pseudorandomness inequality} and \eqref{second pseudorandomness inequality}. We then consider the sets 
\begin{align*} E^{\mathrm{psr}} & = \{x' \in \Z_p^I: E_{I \to x'} \text{ is } (r,\b)\text{-pseudorandom}\},
& E' = \{x' \in E^{\mathrm{psr}}: d(E_{I \to x'}) > \e/2\}, \\
F^{\mathrm{psr}} & = \{y' \in \Z_p^I: F_{I \to y'} \text{ is } (r,\b)\text{-pseudorandom}\},
& F' = \{y' \in F^{\mathrm{psr}}: d(F_{I \to y'}) > \e/2\}. \end{align*} 
We note that any $(x',x'')$ in $E \setminus (E' \times \Z_p^{I^c})$
satisfies $x' \in (E^{\mathrm{psr}})^c$ or $d(E_{I \to x'}) \le \e/2$,
so the density of such $(x',x'')$ in $\Z_p^n$ is at most $\a+\e/2=\e$,
as required for the first conclusion of Theorem \ref{Main result}.
Similarly,  $F \setminus (F' \times \Z_p^{I^c})$ has density at most $\e$.

It remains to show that  $(E'+F') \cap Z_0^I = \emptyset$.
To see this, suppose on the contrary that we have
$(x',y') \in E' \times F'$ with  $x'+y' \in Z_0^I$.
By definition of $E'$ and $F'$ we can apply  
Proposition \ref{Obtaining a non-empty intersection} 
to $E_{I \to x'}$ and $F_{I \to y'}$, thus obtaining
$(x'',y'') \in E_{I \to x'} \times F_{I \to y'}$ with $x''+y'' \in Z_0^{I^c}$.
However, we then have $x=(x',x'') \in E$ and $y=(y',y'') \in F$ with $x+y \in Z_0^n$.
This contradicts our assumption on $(E,F)$,
so  $(E'+F') \cap Z_0^I = \emptyset$, as required.
\end{proof}

\subsection{Pseudorandom summands via correlation}

To complete the proof of Theorem \ref{Main result}, 
it remains to establish Proposition \ref{Obtaining a non-empty intersection},
which we will deduce from a result of Ha\k{z}\l{}a, Holenstein and Mossel \cite{Hazla Holenstein Mossel}. 
We will state it in the simplified setting of two correlated variables,
deferring the more general setting to Section \ref{Section: Extending to several subsets of finite abelian groups}
where we will use it extended to more summands.

Let $(U,V)$ be random variables following some probability distribution $\mathcal{P}$ on $\Omega^2$,
where $\Omega$ is some finite set. The \emph{correlation of $\mathcal{P}$} is 
\[\rho(\mathcal{P})= \sup\{\Cov(\ll(U),\s(V)) : \Var \ll(U) = \Var \s(V) = 1 \},\] 
where the supremum is over functions $\ll,\s: \Omega \to \R$ with $\Var \ll(U) = \Var \s(V) = 1$.
We have $\rho(\mathcal{P}) \le 1$ by the Cauchy-Schwarz inequality,
and we may characterise the equality case as follows.

\begin{lemma} \label{Characterisation of rho = 1}
Let $\Omega$ be a finite set and $(U,V)$ be random variables 
following some probability distribution $\mathcal{P}$ on $\Omega^2$.
Then  $\rho(\cp) = 1$ if and only if there exist non-constant $\ll,\s: \Omega \to \R$ such that $\ll(U)=\s(V)$ with probability $1$.
\end{lemma}

\begin{proof}
Since the quantities $\Cov(\ll(U),\s(V)), \Var \ll(U), \Var \s(V)$ are unchanged after subtracting constants from $\ll$ and $\s$, 
we can view $\rho(\cp)$ as the supremum of $\Cov(\ll(U),\s(V))$ over all $\ll,\s$ with expectation $0$ and variance $1$. 
This supremum is a maximum, as the set of such pairs $(\ll,\s)$ is compact (since the set $\Omega$ is finite).
We consider $(\ll,\s)$ attaining the maximum and apply the Cauchy-Schwarz inequality
 \[  \rho(\cp) = \Cov(\ll(U),\s(V))^2 \le \Var \ll(U) \Var \s(V) = 1.\] 
If $\rho(\cp) = 1$ then we have equality in this inequality, so  $\ll(U)=\s(V)$ with probability $1$;
this gives the required conclusion, as $\ll,\s$ are non-constant (they have non-zero variance).
Conversely, if we can find non-constant $\ll,\s$ satisfying $\ll(U)=\s(V)$ with probability $1$, then $\Var \ll(U) > 0$, 
so after multiplying $\ll$ and $\s$ by $(\Var \ll(U))^{-1/2}$ we get $\Var \ll(U) = \Var \s(V) = 1$ 
as well as $\Cov(\ll(U),\s(V)) = \Var \ll(U) = 1$, so $\rho(\cp) = 1$. \end{proof}

We now state the required result from \cite{Hazla Holenstein Mossel} and deduce Proposition \ref{Obtaining a non-empty intersection}.

\begin{theorem} \label{two-variable case of Theorem 7.1} [\cite{Hazla Holenstein Mossel}, Theorem 7.1, special case] 
For every $\mu > 0$ and $\rho<1$ there exist a positive integer $r$ and some $\beta, c > 0$ such that the following holds. 
Let $(X,Y)$ be a pair of random variables taking values in $\Omega^n \times \Omega^n$ for some finite set $\Omega$, 
such that the pairs $(X_i,Y_i)$ with $i \in [n]$ follow  the same distribution $\mathcal{P}$ on $\Omega^2$ independently.
Let $f,g: \Omega^n \to \{0,1\}$ satisfying the following three assumptions.
\begin{enumerate}[(i)]
\item The sets \{f=1\} and \{g=1\} are $(r,\b)$-pseudorandom,
\item $\E f(X) \ge \mu$ and $\E g(Y) \ge \mu$,
\item $\rho(\cp) \le \rho$.
\end{enumerate} 
Then we have the lower bound 
\begin{equation} \E f(X) g(Y) \ge c. \label{conclusion of the Hazla Holenstein Mossel result} \end{equation}
\end{theorem}

\begin{proof}[Proof of Proposition \ref{Obtaining a non-empty intersection}]
Let $E,F$ be as in Proposition \ref{Obtaining a non-empty intersection}, for some $r,\b,c$ to be determined below.
In Theorem \ref{two-variable case of Theorem 7.1} we take $\Omega$ to be $\Z_p$, take $f=1_E$ and $g=1_F$,
and take $\mathcal{P}$ to be the probability distribution on $\Z_p^2$ 
where each pair $(x,y)$ with $x+y \in Z_0$ has probability $(p|Z_0|)^{-1}$
and all other pairs $(x,y)$ have probability  $0$. 
Setting $\mu = \e$, as $\cp$ has uniform marginals on $\Z_p$
we have $\E f(X) = d(E) \ge \mu$ and $\E g(Y) = d(F) \ge \mu$.

The remaining condition $\rho := \rho(\mathcal{P}) < 1$ follows from Lemma \ref{Characterisation of rho = 1}.
To see this, consider any $\ll,\s: \Z_p \to \R$ such that 
$\ll(U)=\s(V)$ with probability $1$ when $(U,V) \sim \cp$.
As $|Z_0|>1$ we can fix distinct elements $u,v$ of $Z_0$.
For each $x \in \Z_p$, there is a non-zero probability that $(U,V)=(u-x,x)$,
so we must have $\ll(u-x)=\s(x)$. Similarly, we must have $\ll(v-x)=\s(x)$.
Thus $\ll(u-x)=\ll(v-x)$, so $\ll$ is a constant function, and so is $\s$.
Thus $\rho < 1$ by Lemma \ref{Characterisation of rho = 1}.

We can therefore apply Theorem \ref{two-variable case of Theorem 7.1},
which provides the required parameters $r,\b,c$ 
for the assumptions on $E,F$ in Proposition \ref{Obtaining a non-empty intersection}.
Recalling that $S = \{ (x,y) \in \Z_p^n \times \Z_p^n: x+y \in Z_0^n \}$,
the conclusion \eqref{conclusion of the Hazla Holenstein Mossel result} of Theorem \ref{two-variable case of Theorem 7.1}
is equivalent to $|(E \times F) \cap S| \ge c|S|$.
 \end{proof}

\section{Extensions} \label{Section: Extending to several subsets of finite abelian groups}

In the present section we will generalise our result successively
from cyclic groups of prime order to finite abelian groups
and then from two summands to several summands.
The main ideas of the proofs will be the same,
although some non-trivial modifications are needed,
and the second generalisation will require more
of the framework developed in \cite{Hazla Holenstein Mossel}.

\subsection{Extending to finite abelian groups}

We begin with an extension of Theorem \ref{Main result} 
replacing $\Z_p$ for $p$ prime by any finite abelian group $G$.
As mentioned in the introduction, we will need to strengthen our assumption on $Z_0$
to not being contained in any strict coset in $G$.
The following construction shows that this assumption on $Z_0$
is necessary for us to obtain any non-trivial structure for $E,F$.

\begin{example}\label{counterexample for two sets}
Let $G$ be a finite abelian group and let $Z_0 \subset H + \{x\}$, 
where $H$ is a strict subgroup of $G$ and $x \in G$. 
Let $K$ be the quotient group $G/H$ and 
let $\pi:G \to K$ be the projection from $G$ to $K$. 
Then $\pi(Z_0) = \{ \kk \}$ for some $\kk \in K$.

Let $E_K$,$F_K$ be two subsets of $K^n$ such that for any $x_K \in E_K$ and $y_K \in F_K$ we have $x_K+y_K \neq \kk$. 
Consider any $E \subset (\pi^{\otimes n})^{-1}(E_K)$ and $F \subset (\pi^{\otimes n})^{-1}(F_K)$.
Then for any $x \in E$ and $y \in F$ we have $\pi^{\otimes n}(x)+\pi^{\otimes n}(y) \in E_K + F_K$, so $x+y \notin Z_0^n$.
However, the sets $E_K$,$F_K$ are fairly arbitrary: e.g.~if $\kk=0$ then
the only condition is that $E_K$ and $-F_K$ are disjoint.
\end{example}

Now we modify the proof of Theorem \ref{Main result} to prove Theorem \ref{Main result, groups case}.

\begin{proof} [Proof of Theorem \ref{Main result, groups case}]
The proof of Theorem \ref{Main result} also proves Theorem \ref{Main result, groups case},
assuming that we have the corresponding analogue of Proposition \ref{Obtaining a non-empty intersection},
so it suffices to show that the above proof of Proposition \ref{Obtaining a non-empty intersection}
also works replacing `$Z_0 \subset \Z_p$ with $|Z_0|>1$'
by `$Z_0 \subset G$ not contained in any strict coset in $G$'.
We will apply Theorem \ref{two-variable case of Theorem 7.1} similarly to before:
we take $\Omega=G$, $f=1_E$, $g=1_F$ and the distribution $\cp$ on $G^2$
where each pair $(x,y)$ with $x+y \in Z_0$ has probability $(|G||Z_0|)^{-1}$.
Similarly to before this satisfies conditions (i) and (ii) of Theorem \ref{two-variable case of Theorem 7.1},
so it remains to check condition (iii), that is $\rho(\cp)<1$.

We suppose for contradiction that $\rho(\cp)=1$. By Lemma \ref{Characterisation of rho = 1} 
we then have non-constant  $\ll,\s: G \to \R$ such that $\ll(U) = \s(V)$ with probability $1$ when $(U,V) \sim \cp$.
For every $x,y \in G$ with $x+y \in Z_0$ the probability of the event $\{U=x,V=y\}$ is positive, so $\ll(x) = \s(y)$. 
Thus for any $y_1, y_2 \in Z_0$ we have $\s(y_1-x) = \ll(x) = \s(y_2-x)$ for every $x \in G$.
Equivalently, $\s(x) = \s(x+y)$ for all $x \in G$ and all $y \in Z_0-Z_0$. 
Taking $x=0$ and iterating this identity, we see that $\s$ is constant on the subgroup $H$
generated by $Z_0-Z_0$. Furthermore, $Z_0$ is contained in a coset of $H$,
as for any $x \in Z_0$ we have $Z_0-\{x\} \subset Z_0 - Z_0 \subset H$.
By assumption on $Z_0$, we therefore have $H=G$, so $\s$ is constant.
This is a contradiction, so  $\rho(\cp)<1$.
 \end{proof}

As a sanity check, it may be helpful to observe where the above proof fails
(as it must) in the case where $Z_0$ is contained in a strict coset $H+\{x\}$ of $G$,
that is $\pi(Z_0)=\{\kk\}$ with notation as in Example \ref{counterexample for two sets}.
Fix any reals $(a_k: k \in K)$ and consider $\ll,\s: G \to \R$ 
defined by $\ll(x) = a_{\pi(x)}$, $\s(x) = a_{\kk-\pi(x)}$.
These functions are not constant in general,
but satisfy $\ll(x) = \s(y)$ whenever $x+y\in Z_0$.

\subsection{Extending to several summands}

In this subsection we obtain our most general result,
extending the previous result from two summands to any number of summands.
As for two summands, we need to assume that $Z_0 \subset G$ is not contained in any strict coset in $G$.
We start with a construction to show that this assumption is necessary
(in the case of two summands it reduces to a special case of Example \ref{counterexample for two sets}).

\begin{example} \label{general example}
Let $K=G/H$ with $H$ a strict subgroup of $G$ 
and $Z_0 \subset H + \{x\}$ with $\pi(Z_0)=\{\kk\}$ 
be as in Example \ref{counterexample for two sets}.
Define $S:K^n \to K$ by $S(y_1,\dots,y_n) = y_1 + \dots + y_n$.
Fix $a_1, \dots, a_d \in K$ such that $a_1 + \dots + a_d \neq n\kk$,
and define subsets $K_1, \dots, K_d$ of $K^n$ and $E_1,\dots,E_d$ of $G^n$
by $K_i =  S^{-1}(a_i)$ and $E_i = (\pi^{\otimes n})^{-1}(K_i)$ for each $i \in [d]$. 
Then clearly $(K_1 + \dots + K_d) \cap \{\kk\}^n\ = \emptyset$
and so $(E_1 + \dots + E_d) \cap Z_0^n$.
However, the subsets $E_i$ each have density $1/|K|$ but do not satisfy 
the conclusion of Theorem \ref{Main result, groups case, several sets} in general.
\end{example}

The remainder of this subsection will be devoted to the proof of Theorem \ref{Main result, groups case, several sets}. 
We start by stating the appropriate generalisation of Lemma \ref{Simultaneous regularity lemma}
(we omit the proof, as it is a straightforward adaptation of  Lemma \ref{Simultaneous regularity lemma},
considering the sum of energies of the $d$ sets $E_1, \dots, E_d$). 

\begin{lemma} \label{Simultaneous regularity lemma for d sets} 
Let $X$ be a finite set, let $d,r$ be positive integers, and let $\b,\a>0$. 
Then there exists $C = \mathrm{Psr}_d(X,r,\b,\a)$ such that 
for any subsets $E_1,\dots,E_d$ of $X^n$ there exists $I \subset [n]$ with $0<|I|<C$
which simultaneously satisfies 
\begin{equation} \P_{y \in X^I} ((E_j)_{I \to y} \text{ not } (r,\b) \text{-pseudorandom}) \le \a \quad \text{ for every } j \in [d].
\label{not pseudorandom for each of d sets} \end{equation}
\end{lemma}

Next we formulate the appropriate generalisation of Proposition \ref{Obtaining a non-empty intersection}.

\begin{proposition} \label{Obtaining a non-empty intersection for d sets} 
Let $d \ge 2$ be an integer, let $G$ be a finite abelian group,
let $Z_0$ be a subset of $G$ that is not contained in any strict coset in $G$, and let $\e>0$. 
Let 
\[ S = \{ (x^1,\dots,x^d) \in (G^n)^d: x^1+\dots+x^d \in Z_0^n \}. \]
Then there exist $\b>0$, $c>0$, and a positive integer $r$ such that 
whenever $E_1,\dots,E_d$ are $(r,\b)$-pseudorandom subsets of $G^n$ with density at least $\e$
we have  $|S \cap \prod_{i=1}^d E_i| \ge c|S|$.
In particular, there exist $x^1 \in E_1, \dots, x^d \in E_d$ satisfying $x^1 + \dots + x^d \in Z_0^n$.
\end{proposition}

Before proving Proposition \ref{Obtaining a non-empty intersection for d sets}, 
we present the deduction of Theorem \ref{Main result, groups case, several sets},
which is very similar to the case of two summands,
but we nonetheless write it in full for clarity.

\begin{proof}[Proof of Theorem \ref{Main result, groups case, several sets}]
Let $\e>0$, and let $E_1,\dots,E_d$ be subsets of $G^n$ such that $E_1+\dots+E_d$ avoids $Z_0^n$. 
As for two summands, the case where some $E_i$ has density at most $\e$ is trivial,
so we may assume $d(E_i) \ge \e$ for all $i \in [d]$.

We take $r$, $\b$ to be the parameters given by Proposition \ref{Obtaining a non-empty intersection for d sets} applied with $d,G,Z_0,\e/2$. 
Applying Lemma \ref{Simultaneous regularity lemma for d sets} with $X=G$, these $r,\b$ and $\a = \e/2$, 
we obtain a non-empty subset $I$ of $[n]$ with size at most $\mathrm{Psr}_d(G,r, \b, \e/2)$ satisfying  \eqref{not pseudorandom for each of d sets}.
For each $j \in [d]$ we consider
\begin{align*} 
E^{\mathrm{psr}}_j  & = \{x' \in G^I: (E_j)_{I \to x'} \text{ is } (r,\b)\text{-pseudorandom}\}, \text{ and } \\
E'_j & = \{x' \in E^{\mathrm{psr}}_j: d((E_j)_{I \to x'}) > \e/2\}. 
\end{align*}

We note that any $(x',x'')$ in $E_j \setminus (E'_j \times G^{I^c})$
satisfies $x' \in (E^{\mathrm{psr}}_j)^c$ or $d((E_j)_{I \to x'}) \le \e/2$,
so the density of such $(x',x'')$ in $G^n$ is at most $\a+\e/2=\e$.

It remains to show that  $(E'_1+\dots+E'_d) \cap Z_0^I = \emptyset$.
To see this, suppose on the contrary that we have
$((x^1)',\dots,(x^d)') \in E'_1 \times \dots \times E'_d$ with  $(x^1)'+\dots+(x^d)' \in Z_0^I$.
By definition of $E'_1,\dots,E'_d$ we can apply  
Proposition \ref{Obtaining a non-empty intersection for d sets} 
to $(E_1)_{I \to (x^1)'}, \dots, (E_d)_{I \to (x^d)'}$, thus obtaining
$((x^1)'',\dots,(x^d)'') \in (E_1)_{I \to (x^1)'} \times \dots \times (E_d)_{I \to (x^d)'}$ 
with $(x^1)''+\dots+(x^d)'' \in Z_0^{I^c}$.
However, we then have $x^j=((x^j)',(x^j)'') \in E_j$ for each $j \in [d]$
with $x^1+\dots+x^d \in Z_0^n$.
This contradicts our assumption on $(E_1,\dots,E_d)$,
so  $(E'_1+\dots+E'_d) \cap Z_0^I = \emptyset$, as required.
 \end{proof}

To complete the proof of Theorem  \ref{Main result, groups case, several sets},
it remains to establish Proposition \ref{Obtaining a non-empty intersection for d sets},
via the result of \cite{Hazla Holenstein Mossel}, which we will now state in more generality.
Let $(U_1,\dots,U_d)$ be random variables following some probability distribution $\mathcal{P}$ on $\Omega^d$,
where $\Omega$ is some finite set. Let
\[\rho(\mathcal{P})= \max_{j \in [d]} \sup\{\Cov(\ll(U_j),\s(\overline{U_j})) : \Var \ll(U_j) = \Var \s(\overline{U_j}) = 1 \},\] 
where we write $\overline{U_j}$ for $(U_1, \dots, U_{j-1}, U_{j+1}, \dots, U_d)$
and consider functions $\ll: \Omega \to \R$ and $\s: \Omega^{[d] \setminus\{j\}} \to \R$.
We now state the required result from \cite{Hazla Holenstein Mossel} 
and deduce Proposition \ref{Obtaining a non-empty intersection for d sets}.

\begin{theorem} \label{d-variable case of Theorem 7.1} [\cite{Hazla Holenstein Mossel}, Theorem 7.1, special case] 
For every integer $d \ge 2$, $\mu > 0$ and $\rho<1$ there exists a positive integer $r$ and some $\beta, c > 0$ such that the following holds. 

Let $(X^1,\dots,X^d)$ be random variables taking values in $(\Omega^n)^d$ for some finite set $\Omega$, 
such that $(X^1_i,\dots,X^d_i)$ with $i \in [n]$ follow the same distribution $\cp$ on $\Omega^d$ independently.

Suppose $\rho(\cp) \le \rho$ and $f_1,\dots,f_d: \Omega^n \to \{0,1\}$ are such that for every $j \in [d]$
the set $\{f_j=1\} \subset \Omega^n$ is $(r,\b)$-pseudorandom and $\E f_j(X^j) \ge \mu$.
Then $ \E f_1(X^1) \dots f_d(X^d) \ge c$. 
\end{theorem}

\begin{proof}[Proof of Proposition \ref{Obtaining a non-empty intersection for d sets}]
Let $E_1, \dots, E_d$ be as in Proposition \ref{Obtaining a non-empty intersection for d sets}. 
In Theorem \ref{d-variable case of Theorem 7.1} we take $\Omega$ to be $G$
and $\mathcal{P}$ to be the probability distribution on $G^d$ 
where each $(x^1, \dots, x^d)$ with $x^1+\dots+x^d \in Z_0$ has probability $(|G|^{d-1}|Z_0|)^{-1}$ 
and all other $d$-tuples $(x^1, \dots, x^d)$ have probability $0$.
Each marginal of $\cp$ is uniform, so taking $f_j = 1_{E_j}$ 
we have $\E f_j(X^j) = d(E_j) \ge \mu := \e$ for every $j \in [d]$.

As in the proof of Proposition \ref{Obtaining a non-empty intersection},
the proposition will follow from Theorem \ref{d-variable case of Theorem 7.1}  
once we verify the remaining condition $\rho(\cp)<1$.
We can use the characterisation from Lemma \ref{Characterisation of rho = 1} 
even for $\ll: G \to \R$ and $\s: G^{[d] \setminus \{i\}} \to \R$ with $i \in [d]$, 
since its proof did not rely on the domains on which $\ll,\s$ are defined. 
The roles of the $d$ coordinates are interchangeable, 
so it suffices to show that if $\ll: G \to \R$ and $\s: G^{d-1} \to \R$ satisfy 
$\ll(U_1) = \s(U_2,\dots,U_d)$ with probability $1$ when $(U_1,\dots,U_d) \sim \mathcal{P}$ then $\ll, \s$ are constant. 

Fix any $(x^3,\dots,x^d) \in G^{d-2}$ and consider the event $A = \{ U_3 = x^3, \dots, U_d = x^d\}$.
(We can assume $d>2$, or regard $A$ as the trivial event that always holds if $d=2$.)
Then $\cp(A)>0$, so we can consider the conditional distribution 
$\cp_{(x^3,\dots,x^d)}$ of $(U^1,U^2)$ under $\mathcal{P}$ conditioned on $A$,
which satisfies $\ll(U_1) = \s(U_2,x^3,\dots,x^d)$ with probability $1$.

We note that $\mathcal{P}_{(x^3,\dots,x^d)}$ is the distribution on $G^2$
that is uniformly distributed on pairs $(x^1,x^2)$ with $x^1 + x^2$ 
in the translate $Z_0 - \{x^3 + \dots + x^d\}$ of $Z_0$.
Applying the $d=2$ case (which we proved in  Proposition \ref{Obtaining a non-empty intersection})
to $\mathcal{P}_{(x^3,\dots,x^d)}$, which is valid as  $Z_0 - \{x^3 + \dots + x^d\}$ is not contained
in any strict coset in $G$, we have $\rho(\mathcal{P}_{(x^3,\dots,x^d)})<1$.

We then define $\s_{(x^3,\dots,x^d)}: G \to \R$ by $\s_{(x^3,\dots,x^d)}(x^2) = \s(x^2,x^3,\dots,x^d)$
and apply Lemma \ref{Characterisation of rho = 1} to $\ll$ and $\s_{(x^3,\dots,x^d)}$.
As $\ll(U_1) = \s_{(x^3,\dots,x^d)}(U_2)$ with probability $1$ for $(U_1,U_2) \sim\mathcal{P}_{(x^3,\dots,x^d)}$,
we deduce that  $\ll$ and $\s_{(x^3,\dots,x^d)}$ are constant, always taking some fixed value $c$.
 
It remains to show that $\s$ is constant. 
Consider any $(x^2,x^3,\dots,x^d) \in G^{d-1}$ and $x^1 \in Z_0 - \{x_2+x_3+\dots+x_d\}$.
Then $(x^1,x^2)$ is in the support of $(U_1,U_2) \sim\mathcal{P}_{(x^3,\dots,x^d)}$,
so $c = \ll(x^1) = \s_{(x^3,\dots,x^d)}(x^2) = \s(x^2,x^3,\dots,x^d)$.
\end{proof}

\section{Further discussion}\label{Section: Further discussion}

Our first open problem is to improve the bound in our main theorems, starting with the simplest setting.

\begin{question}

Can we for each prime $p$ require $|I| \le O(\log(\e^{-1}))$ in Theorem \ref{Main result} ?

\end{question}

Such a bound would be optimal, as is shown by the example of 
$Z_0 = \{0,1\}$, $\e = p^{-(k+1)}$, $E = (\Z_p^k \setminus \{0,1\}^k) \times \Z_p^{n-k}$, 
and $F = \{0\}^k \times \Z_p^{n-k}$ for some integer $k$.
Indeed, if $E',F'$ satisfy the conclusion of Theorem \ref{Main result} for some non-empty $I \subset [n]$
then $F' \ne \emptyset$ and $E' \ne \Z_p^I$, so $E_0 := E \setminus (E' \times \Z_p^{I^c})$
satisfies $\e \ge d(E_0) \ge  p^{-|I|} - (p/2)^{-k}$, giving $|I| = \Omega(k)$.

Our second problem is to consider other directions of generalising Theorem \ref{Main result},
besides the extensions to finite abelian groups and more summands considered in this paper.
Our general setting (with two inputs for simplicity)
is an arbitrary function  $f: X \times Y \to Z$ and its tensor product $f^{\otimes n}: X^n \times Y^n \to Z^n$, where $X,Y,Z$ are finite sets.
We would like to characterise pairs $(f,Z_0)$ with $Z_0 \subset Z$ 
satisfying the following statement analogous to that in Theorem \ref{Main result}:
for every $\e>0$ there is some $C=C(f,\e)$ 
such that for any $E \subset X^n$ and $F \subset Y^n$ with $f^{\otimes n}(E,F) \cap Z_0^n  = \emptyset$
there exist $I \subset [n]$ with $0<|I|<C$ and $E' \subset X^I$, $F' \subset Y^I$ satisfying 
\[ |E \setminus (E' \times X^{I^c})| \le \e |X^n|,
\quad |F \setminus (F' \times Y^{I^c})| \le \e |Y^n|,
\quad f^{\otimes I} (E',F') \cap Z_0^I = \emptyset. \]

The following result gives a general condition on $(f,Z_0)$ 
that guarantees the existence of the desired sets $I,E',F'$, 
for the trivial reason that there are almost complete sets $E_n$ and $F_n$
such that $f^{\otimes n}(E_n,F_n) \cap Z_0^n  = \emptyset$.
(To be formal, consider any $\e>0$, let $n'$ be such that 
we can find $E_{n'}$, $F_{n'}$ as below with densities at least $1-\eps$,
and take $I = [n], E'=E,F'=F$ if $n<n'$, or $I = [n'], E' = E_{n'}, F' = F_{n'}$ if $n \ge n'$.)

\begin{proposition} \label{construction with empty intersection but very dense sets}
Let $X,Y,Z$ be finite sets, let $f: X \times Y \to Z$ and let $Z_0 \subset Z$.
Suppose that there exist $A \subset X$ and $B \subset Y$ such that 
$f(A \times B) \cap Z_0 = \emptyset$  and $|A|/|X| + |B|/|Y| > 1$.

Then there exist subsets $E_n \subset X^n$ and $F_n \subset Y^n$ for every positive integer $n$ 
such that $f^n(E_n \times F_n) \cap Z_0^n = \emptyset$ and $\min(d(E_n),d(F_n)) \to 1$ as $n \to \infty$.
\end{proposition}

The condition in Proposition \ref{construction with empty intersection but very dense sets}
is satisfied by some natural functions, such as the minimum function $[2k] \times [2k] \to [2k]$
for some $k \ge 2$ with $Z_0 = [2k] \setminus [k+1]$. 
On the other hand, it contrasts heavily with the case that $f:G \times G \to G$
is addition on a finite abelian group $G$, as for any $A,B \subset G$ with $A+B \ne G$
we have $A$ disjoint from some translate of $-B$, so $d(A)+d(B) \le 1$,
and similarly $d(E)+d(F) \le 1$ for any $E,F \subset G^n$ with $E+F \ne G^n$.

\begin{proof} 
Let $r,s,n$ be positive integers such that $rs \le n$. We consider the `tribe-like' sets
\begin{align*} E_n & = \{x_1 \in A \vee \dots \vee x_r \in A\} \wedge \dots \wedge \{x_{r(s-1)+1} \in A \vee \dots \vee x_{rs} \in A\}, \\ 
F_n & = \{y_1 \in B \wedge \dots \wedge y_r \in B\} \vee \dots \vee \{y_{r(s-1)+1} \in B \wedge \dots \wedge y_{rs} \in B\}. 
\end{align*} 
We claim that $f^{\otimes n}(E_n \times F_n) \cap Z_0^n = \emptyset$. 
To see this, consider any $(x,y) \in E_n \times F_n$.
By definition of $F_n$, there is some $s' \in [s]$ 
such that $y_i \in B$ for all $i \in [r(s'-1)+1,rs']$.
By definition of $E_n$, there is some $i \in [r(s'-1)+1,rs']$
such that $x_i \in A$, so $(x_i, y_i) \in A \times B$.
As $f(A \times B) \cap Z_0 = \emptyset$,
we deduce $f^{\otimes n}(x,y) \notin Z_0^n$, so the claim holds.

It remains to estimate the densities of $E_n$ and $F_n$. We have
\begin{align*} 
1-d(E_n) & =  1- (1-(1-|A|/|X|)^r)^s \le  s(1-|A|/|X|)^r, \\
1-d(F_n) & = (1-(|B|/|Y|)^r)^s \le \exp(-s(|B|/|Y|)^r).
\end{align*} 
As $1-|A|/|X| < |B|/|Y|$, for any $\e>0$ we can choose $r$ 
such that $(1-|A|/|X|)^r (|Y|/|B|)^r < \e$.
Then choosing $s = \lceil (|Y|/|B|)^r \log(\e^{-1}) \rceil$
gives $1-d(F_n) \le \e$ and $1-d(E_n) \le 2\e \log(\e^{-1})$,
which can be both arbitrarily small, as required.
\end{proof}

For further insight on the problem for general functions $f$,
it seems natural to generalise from addition in abelian groups,
to multiplication in general groups, and then further to latin squares,
namely functions $f: X \times X \to X$ (for some finite set $X$) such that for every $x,y$ in $X$
there is a unique $z$ such that $f(x,z)=y$ 
and  a unique $z'$ such that $f(z',x)=y$.

\begin{question}
Does some analogue of  Theorem \ref{Main result} hold for latin squares?
\end{question}

\end{document}